\documentclass{article}%
\usepackage{amsfonts}
\usepackage{amsmath}
\usepackage{amssymb}
\usepackage{graphicx}%
\newtheorem{theorem}{Theorem}

\newtheorem{corollary}[theorem]{Corollary}

\newtheorem{proposition}[theorem]{Proposition}

\newenvironment{proof}[1][Proof]{\noindent\textbf{#1.} }{\ \rule{0.5em}{0.5em}}
\begin{document}

\title{$A$-Poisson structures }
\author{Basile Guy Richard BOSSOTO$^{(1)}$, Eug\`{e}ne OKASSA$^{(2)}$\\Universit\'{e} Marien NGOUABI, Facult\'{e} des Sciences,\\D\'{e}partement de Math\'{e}matiques\\B.P.69 - BRAZZAVILLE- (Congo)\\E-mail: $^{(1)}$bossotob@yahoo.fr, $^{(2)}$eugeneokassa@yahoo.fr}
\date{}
\maketitle

\begin{abstract}
Let $M$ be a paracompact differentiable manifold, $A$ a local algebra and
$M^{A}$ a manifold of infinitely near points on $M$ of kind $A$. We define the
notion of $A$-Poisson manifold on $M^{A}$. We show that when $M$ is a Poisson
manifold, then $M^{A}$ is an $A$-Poisson manifold. We also show that if
$(M,\Omega)$ is a symplectic manifold, the structure of $A$-Poisson manifold
on $M^{A}$ defined by $\Omega^{A}$ coincide with the prolongation on $M^{A}$
of the Poisson structure on $M$ defined by the symplectic form $\Omega$.

\end{abstract}



\section{Introduction}

In what follows,$\ M$ denotes a paracompact differentiable manifold,
$C^{\infty}(M)$ the algebra of smooth functions on $M$, $A$ a local algebra
(in the sense of Andr\'{e} Weil) i.e a real commutative algebra with unit, of
finite dimension, and with an unique maximal ideal $\mathfrak{m}$ of
codimension $1$ over $%
\mathbb{R}
$. In this case, there exists an integer $h$ such that $\mathfrak{m}%
^{h+1}=(0)$ and $\mathfrak{m}^{h}\neq(0)$. The integer $h$ is the height of
$A$. Also we have $A=%
\mathbb{R}
\oplus\mathfrak{m}$.

For example the algebra of dual numbers
\[
\mathbb{D}=%
\mathbb{R}
\left[  T\right]  /(T^{2})
\]
is a local algebra with height $1$.

We recall that a near point of $x\in M$ of kind $A$ is a morphism of algebras
\[
\xi:C^{\infty}(M)\longrightarrow A
\]
such that
\[
\left[  \xi(f)-f(x)\right]  \in\mathfrak{m}%
\]
for any $f\in C^{\infty}(M)$. We denote $M_{x}^{A}$ the set of near points of
$x\in M$ of kind $A$ and
\[
M^{A}=\bigcup\limits_{x\in M}M_{x}^{A}%
\]
the manifold of infinitely near points on $M$ of kind $A$ \cite{wei}.

We have $%
\mathbb{R}
^{A}=A$, $M^{\mathbb{D}}=TM$ where $TM$ is the tangent bundle of $M$.

When the dimension of $M$ is $n$,\ then the dimension of $M^{A}$ is
$n\times\dim(A)$ \cite{wei}. Let $(U,\varphi)$ be a local chart with local
coordinates $(x_{1},x_{2},...,x_{n})$. The application
\[
U^{A}\longrightarrow A^{n},\xi\longmapsto(\xi(x_{1}),\xi(x_{2}),...,\xi
(x_{n})),
\]
is a bijection from $U^{A}$ to an open of $A^{n}$. Thus $M^{A}$ is an
$A$-manifold of dimension $n$.

The set, $C^{\infty}(M^{A},A)$, of smooth functions on $M^{A}$ with values in
$A$ is a commutative algebra with unit over $A$.

For any $f\in C^{\infty}(M)$, the application%
\[
f^{A}:M^{A}\longrightarrow A,\xi\longmapsto\xi(f),
\]
is smooth and the application%
\[
C^{\infty}(M)\longrightarrow C^{\infty}(M^{A},A),f\longmapsto f^{A},
\]
is a monomophism of algebras.

The following assertions are equivalent \cite{bos}:

\begin{enumerate}
\item $X$ is a derivation of $C^{\infty}(M^{A})$ i.e. $X$ is a vector field on
$M^{A}$;

\item $X:C^{\infty}(M)\longrightarrow C^{\infty}(M^{A},A)$ is a $%
\mathbb{R}
$-linear application such that, for any $f,g\in C^{\infty}(M)$,%
\[
X(fg)=X(f)\cdot g^{A}+f^{A}\cdot X(g)
\]
i.e. $X$ is a derivation from $C^{\infty}(M)$ to $C^{\infty}(M^{A},A)$ with
respect the module structure
\[
C^{\infty}(M^{A},A)\times C^{\infty}(M)\longrightarrow C^{\infty}%
(M^{A},A),(F,f)\longmapsto F\cdot f^{A}\text{.}%
\]

\end{enumerate}

Thus the set, $\mathfrak{X}(M^{A})$, of vector fields on $M^{A}$ considered as
derivations of $C^{\infty}(M)$ into $C^{\infty}(M^{A},A)$ is a module over
$C^{\infty}(M^{A},A)$.

When
\[
\theta:C^{\infty}(M)\longrightarrow C^{\infty}(M)
\]
is a vector field on $M$, then the application
\[
\theta^{A}:C^{\infty}(M)\longrightarrow C^{\infty}(M^{A},A),f\longmapsto
\left[  \theta(f)\right]  ^{A}\text{,}%
\]
is a vector field on $M^{A}$. We say that the vector field $\theta^{A}$ is the
prolongation to $M^{A}$ of the vector field $\theta$.

If $X$ is a vector field on $M^{A}$, considered as a derivation of $C^{\infty
}(M)$ into $C^{\infty}(M^{A},A)$, then there exists, \cite{bos}, an unique
derivation
\[
\widetilde{X}:C^{\infty}(M^{A},A)\longrightarrow C^{\infty}(M^{A},A)
\]
such that

\begin{enumerate}
\item $\widetilde{X}$ is $A$-linear;

\item $\widetilde{X}\left[  C^{\infty}(M^{A})\right]  \subset C^{\infty}%
(M^{A})$;

\item $\ \widetilde{X}(f^{A})=X(f)$ for any $f\in C^{\infty}(M)$.
\end{enumerate}

Let $(a_{\alpha})_{\alpha=1,...,r}$ be a basis of $A$ and $(a_{\alpha}^{\ast
})_{\alpha=1,...,r}$ be the dual basis.

If
\[
Y:C^{\infty}(M^{A},A)\longrightarrow C^{\infty}(M^{A},A)
\]
is an $A$-linear derivation such that
\[
Y(f^{A})=\ \widetilde{X}(f^{A})
\]
for any $f\in C^{\infty}(M)$, then
\[
Y\left[  C^{\infty}(M^{A})\right]  \subset C^{\infty}(M^{A})
\]
since
\[
Y(a_{\alpha}^{\ast}\circ f^{A})=\ \widetilde{X}(a_{\alpha}^{\ast}\circ
f^{A})\in C^{\infty}(M^{A})
\]
for any $\alpha=1,2,..,r$. Thus, \cite{bos}, $Y=\ \widetilde{X}$.

The application
\[
\left[  ,\right]  :\mathfrak{X}(M^{A})\times\mathfrak{X}(M^{A})\longrightarrow
\mathfrak{X}(M^{A}),(X,Y)\longmapsto\widetilde{X}\circ Y-\widetilde{Y}\circ
X\text{,}%
\]
is $A$-bilinear and defines a structure of $A$-Lie algebra on $\mathfrak{X}%
(M^{A})$ \cite{bos}.

If we denote $Der\left[  C^{\infty}(M^{A},A)\right]  $, the $C^{\infty}%
(M^{A},A)$-module of derivations of $C^{\infty}(M^{A},A)$, then the application%

\[
\mathfrak{X}(M^{A})\longrightarrow Der\left[  C^{\infty}(M^{A},A)\right]
,X\longmapsto\widetilde{X},
\]
is a morphism of $A$-Lie algebras \cite{bos}.

For any $p\in\mathbb{N}$,
\[
\Lambda^{p}(M^{A},A)=\mathcal{L}_{sks}^{p}\left[  \mathfrak{X}(M^{A}%
),C^{\infty}(M^{A},A)\right]
\]
denotes the $C^{\infty}(M^{A},A)$-module of skew-symmetric multilinear forms
of degree $p$ on $\mathfrak{X}(M^{A})$. We say that $\Lambda^{p}(M^{A},A)$ is
the $C^{\infty}(M^{A},A)$-module of differential $A$-forms of degree $p$ on
$M^{A}$. We have
\[
\Lambda^{0}(M^{A},A)=C^{\infty}(M^{A},A)\text{.}%
\]
We denote
\[
\Lambda(M^{A},A)=\bigoplus\limits_{p=0}^{n}\Lambda^{p}(M^{A},A)\text{.}%
\]

If $\omega$ is a differential form of degree $p$ on $M$, then there exists an
unique differential $A$-form of degree $p$ on $M^{A}$\ such that
\[
\omega^{A}(\theta_{1}^{A},\theta_{2}^{A},...,\theta_{p}^{A})=\left[
\omega(\theta_{1},\theta_{2},...,\theta_{p})\right]  ^{A}%
\]
for any vector fields $\theta_{1},\theta_{2},...,\theta_{p}\ \in\mathfrak{X}($
$M)$. \ We say that the differential $A$-form $\omega^{A}$ is the prolongation
to $M^{A}$ of the differential form $\omega$ \cite{mor}, \cite{ok3}.

When
\[
d:\Lambda(M)\longrightarrow\Lambda(M)
\]
is the exterior diffentiation operator, we denote
\[
d^{A}:\Lambda(M^{A},A)\longrightarrow\Lambda(M^{A},A)
\]
the cohomology operator associated to the representation
\[
\mathfrak{X}(M^{A})\longrightarrow Der\left[  C^{\infty}(M^{A},A)\right]
,X\longmapsto\widetilde{X}\text{.}%
\]
We recall that for $\eta\in\Lambda^{p}(M^{A},A)$, we have%
\begin{align*}
&  (d^{A}\eta)(X_{1},X_{2},...,X_{p+1})\\
&  =\overset{p+1}{\underset{i=1}{\sum}}(-1)^{i-1}\widetilde{X_{i}}\left[
\eta(X_{1},X_{2},...,\widehat{X_{i}},...,X_{p+1})\right] \\
&  +\underset{1\leq i<j\leq p+1}{\sum}(-1)^{i+j}\ \eta(\left[  X_{i}%
,X_{j}\right]  ,X_{1},...,\widehat{X_{i}},\ ...,\widehat{X_{j}},...,X_{p+1})
\end{align*}
for any vector fields $X_{1},X_{2},...,X_{p+1}$ on $M^{A}$, where
$\widehat{X_{i}}$ means that the term $X_{i}$ is omitted.

The application%
\[
d^{A}:\Lambda(M^{A},A)\longrightarrow\Lambda(M^{A},A)
\]
is $A$-linear and
\[
d^{A}(\omega^{A})=(d\omega)^{A}%
\]
for any $\omega\in\Lambda(M)$ \cite{bos}. It is obvious that if%
\[
d\omega=0\text{,}%
\]
then
\[
d^{A}(\omega^{A})=0\text{.}%
\]

\bigskip Let $(M,\omega)$ be a symplectic manifold. Then the manifold $M$ is a
Poisson manifold i.e. the algebra $C^{\infty}(M)$ carries a structure of
Poisson algebra. For any linear form
\[
\psi:A\longrightarrow%
\mathbb{R}
\text{,}%
\]
the differential form $\psi\circ\omega^{A}$ is not necessary a symplectic form
on $M^{A}$. That means that the prolongation $\omega^{A}$ does not always
induce a structure of Poisson on $M^{A}$. In effect, let $m$ be the maximal
ideal of a local algebra $A$,
\[
ann(m)=\left\{  a\in A/a\cdot x=0\text{ for any }x\in m\right\}
\]
and%
\[
\mu_{A}:A\times A\longrightarrow A,(a,b)\longmapsto a\cdot b,
\]
the multiplication on $A$. Then there exists a linear form $\psi
:A\longrightarrow%
\mathbb{R}
$ such that the bilinear symmetric form
\[
\psi\circ\mu_{A}:A\times A\longrightarrow%
\mathbb{R}
\]
is nondegenerated if and only if $\dim\left[  ann(m)\right]  =1$\cite{ok3}.

When $\left(  M,\omega\right)  $ is a symplectic manifold and $\psi\in
A^{\ast}$ a linear form on $A$, then the scalar $2$-form $\psi\circ\omega^{A}$
is a symplectic form on $M^{A}$ if and only if $\dim\left[  ann(m)\right]  =1$
and $\psi\left[  ann(m)\right]  \neq0$ : it is the case when $A=%
\mathbb{R}
\left[  T_{1},...,T_{s}\right]  /[T_{1}^{k_{1}},...,T_{s}^{k_{s}}]$. Thus,
when $\left(  M,\omega\right)  $ is a symplectic manifold, we cannot obtain a
Poisson structure on $M^{A}$ which comes from the prolongation of $\omega$
when $\dim\left[  \text{ann}(m)\right]  \neq1$. For example, it is the case
when $A=%
\mathbb{R}
\left[  T_{1},T_{2}\right]  /(T_{1},T_{2})^{2}$.

In this paper, we do not study the structures of $A$-manifolds but we study
the structures on $M^{A}$ as an $A$-manifold. When $M$ is a manifold, the
basic algebra of $M$ is $C^{\infty}(M)$. As $\mathfrak{X}(M^{A})$ is a
$C^{\infty}(M^{A},A)$-module, considered as the set of derivations of
$C^{\infty}(M)$ to $C^{\infty}(M^{A},A)$, and a Lie algebra over $A$, and as
$M^{A}$ is an $A$-manifold, that means that the basic algebra of $M^{A}$ is
$C^{\infty}(M^{A},A)$: thus the natural space for studying Poisson structures
on $M^{A}$ is $C^{\infty}(M^{A},A)$ but not $C^{\infty}(M^{A})$. When $\left(
M,\omega\right)  $ is a symplectic manifold, we will show \ that $\left(
M^{A},\omega^{A}\right)  $ is a symplectic $A$-manifold.

The main goal of this paper is to define the notion of $A$-Poisson structures
on $M^{A}$ and to show that if a manifold $M$ is a Poisson manifold, then
$M^{A}$ admits an $A$-Poisson structure.\ We also show that if $M$ is a
symplectic manifold, then $M^{A}$ admits an $A$-Poisson structure such that
this structure coincide with the structure of $A$-Poisson manifold on $M^{A}$
deduced by the structure of Poisson manifold on $M$ defined by the symplectic form.

\section{$A$-Poisson structures}

We recall that a Poisson structure on a differentiable manifold $M$ is due to
the existence of a bracket $\left\{  ,\right\}  $ on $C^{\infty}(M)$ such that
the pair $\left(  C^{\infty}(M),\left\{  ,\right\}  \right)  $ is a real Lie
algebra and%
\[
\left\{  f,g\cdot h\right\}  =\left\{  f,g\right\}  \cdot h+g\cdot\left\{
f,h\right\}
\]
for any $f,g,h\in C^{\infty}(M)$. In this case we say that $M$ is a Poisson
manifold and $C^{\infty}(M)$ is a Poisson algebra.

We will say that the $A$-algebra $C^{\infty}(M^{A},A)$ is a Poisson
$A$-algebra if there exists a bracket $\left\{  ,\right\}  $ on $C^{\infty
}(M^{A},A)$ such that the pair $\left(  C^{\infty}(M^{A},A),\left\{
,\right\}  \right)  $ is a Lie algebra over $A$ satisfaying
\[
\left\{  \varphi,\psi_{1}\cdot\psi_{2}\right\}  =\left\{  \varphi,\psi
_{1}\right\}  \cdot\psi_{2}+\psi_{1}\cdot\left\{  \varphi,\psi_{2}\right\}
\]
for any $\varphi,\psi_{1},\psi_{2}\in C^{\infty}(M^{A},A)$. When $C^{\infty
}(M^{A},A)$ is a Poisson $A$-algebra, we will say that the manifold $M^{A}$ is
a $A$-Poisson manifold or $M^{A}$ admits an $A$-Poisson structure.

\subsection{Structure of $A$-Poisson manifold on $M^{A}$ when $M$ is a Poisson
manifold}

In this part, $M$ is a Poisson manifold with bracket $\left\{  ,\right\}  $.
In this case, for any $f\in C^{\infty}(M)$, the application%
\[
ad(f):C^{\infty}(M)\longrightarrow C^{\infty}(M),g\longmapsto\left\{
f,g\right\}  ,
\]
is a vector field on $M$ and, for any $g\in C^{\infty}(M)$, we get
\[
ad(fg)=f\cdot ad(g)+g\cdot ad(f)\text{.}%
\]

For any $f\in C^{\infty}(M)$, let
\[
\left[  ad(f)\right]  ^{A}:C^{\infty}(M)\longrightarrow C^{\infty}%
(M^{A},A),g\longmapsto\left\{  f,g\right\}  ^{A},
\]
be the prolongation of the vector field $ad(f)$ and let
\[
\widetilde{\left[  ad(f)\right]  ^{A}}:C^{\infty}(M^{A},A)\longrightarrow
C^{\infty}(M^{A},A)
\]
be the unique $A$-linear derivation such that
\begin{align*}
\widetilde{\left[  ad(f)\right]  ^{A}}(g^{A})  &  =\left[  ad(f)\right]
^{A}(g)\\
&  =\left\{  f,g\right\}  ^{A}%
\end{align*}
for any $g\in C^{\infty}(M)\cite{bos}$.

\begin{proposition}
For any $\varphi\in C^{\infty}(M^{A},A)$, the application%
\[
\tau_{\varphi}:C^{\infty}(M)\longrightarrow C^{\infty}(M^{A},A),f\longmapsto
-\widetilde{[ad(f)]^{A}}(\varphi),
\]
is a vector field on $M^{A}$.
\end{proposition}

\begin{proof}
It is obvious that $\tau_{\varphi}$ is linear. For any $f,g\in C^{\infty}(M)$,
we have%
\begin{align*}
\tau_{\varphi}(fg)  &  =-\widetilde{[ad(fg)]^{A}}(\varphi)\\
&  =-\widetilde{[f\cdot ad(g)+g\cdot ad(f)]^{A}}(\varphi)\\
&  =f^{A}\cdot\left(  -\widetilde{[ad(g)]^{A}}\right)  (\varphi)+g^{A}%
\cdot\left(  -\widetilde{[ad(f)]^{A}}\right)  (\varphi)\\
&  =\left(  -\widetilde{[ad(f)]^{A}}\right)  (\varphi)\cdot g^{A}+f^{A}%
\cdot\left(  -\widetilde{[ad(f)]^{A}}\right)  (\varphi)\\
&  =\tau_{\varphi}(f)\cdot g^{A}+f^{A}\cdot\tau_{\varphi}(g)\text{.}%
\end{align*}
That ends the proof.
\end{proof}

For any $\varphi\in C^{\infty}(M^{A},A)$, we denote
\[
\widetilde{\tau_{\varphi}}:C^{\infty}(M^{A},A)\longrightarrow C^{\infty}%
(M^{A},A)
\]
the unique $A$-linear derivation such that%
\[
\widetilde{\tau_{\varphi}}(f^{A})=\tau_{\varphi}(f)
\]
for any $f\in C^{\infty}(M)$.

For $f\in C^{\infty}(M)$, we verify that%
\[
\widetilde{\tau_{f^{A}}}=\widetilde{\left[  ad(f)\right]  ^{A}}\text{.}%
\]

\begin{proposition}
For $\varphi,\psi\in C^{\infty}(M^{A},A)$ and for $a\in A$, we have
\begin{align*}
\widetilde{\tau_{\varphi+\psi}}  &  =\widetilde{\tau_{\varphi}}+\widetilde
{\tau_{\psi}}\text{;}\\
\widetilde{\tau_{a\cdot\varphi}}  &  =a\cdot\widetilde{\tau_{\varphi}}%
\text{;}\\
\widetilde{\tau_{\varphi\cdot\psi}}  &  =\varphi\cdot\widetilde{\tau_{\psi}%
}+\psi\cdot\widetilde{\tau_{\varphi}}\text{.}%
\end{align*}

\end{proposition}

\begin{proof}
For $\varphi,\psi\in C^{\infty}(M^{A},A)$, $\widetilde{\tau_{\varphi}%
}+\widetilde{\tau_{\psi}}$ is an $A$-linear derivation. For any $f\in
C^{\infty}(M)$, we get%
\begin{align*}
(\widetilde{\tau_{\varphi}}+\widetilde{\tau_{\psi}})(f^{A})  &  =(\widetilde
{\tau_{\varphi}})(f^{A})+(\widetilde{\tau_{\psi}})(f^{A})\\
&  =\left(  -\widetilde{[ad(f)]^{A}}\right)  (\varphi)+\left(  -\widetilde
{[ad(f)]^{A}}\right)  (\psi)\\
&  =\left(  -\widetilde{[ad(f)]^{A}}\right)  (\varphi+\psi)\\
&  =(\widetilde{\tau_{\varphi+\psi}})(f^{A})\text{.}%
\end{align*}
We deduce that
\[
\widetilde{\tau_{\varphi+\psi}}=\widetilde{\tau_{\varphi}}+\widetilde
{\tau_{\psi}}\text{.}%
\]

For $\varphi\in C^{\infty}(M^{A},A)$, $a\in A$ and $f\in C^{\infty}(M)$, we
have%
\begin{align*}
(a\cdot\widetilde{\tau_{\varphi}})(f^{A})  &  =a\cdot(\widetilde{\tau
_{\varphi}})(f^{A})\\
&  =a\cdot\left(  -\widetilde{[ad(f)]^{A}}\right)  (\varphi)\\
&  =\left(  -\widetilde{[ad(f)]^{A}}\right)  (a\cdot\varphi)\\
&  =\left(  \widetilde{\tau_{a\cdot\varphi}}\right)  (f^{A})\text{.}%
\end{align*}
We deduce that%
\[
\widetilde{\tau_{a\cdot\varphi}}=a\cdot\widetilde{\tau_{\varphi}}\text{.}%
\]
\ \ \ \ 

For $\varphi,\psi\in C^{\infty}(M^{A},A)$, $\varphi\cdot\widetilde{\tau_{\psi
}}+\psi\cdot\widetilde{\tau_{\varphi}}$ is an $A$-linear derivation. For any
$f\in C^{\infty}(M)$, we get
\begin{align*}
\left[  \varphi\cdot\widetilde{\tau_{\psi}}+\psi\cdot\widetilde{\tau_{\varphi
}}\right]  (f^{A})  &  =\varphi\cdot\widetilde{\tau_{\psi}}(f^{A})+\psi
\cdot\widetilde{\tau_{\varphi}}(f^{A})\\
&  =\varphi\cdot\left(  -\widetilde{[ad(f)]^{A}}\right)  (\psi)+\psi
\cdot\left(  -\widetilde{[ad(f)]^{A}}\right)  (\varphi)\\
&  =\left(  -\widetilde{[ad(f)]^{A}}\right)  (\varphi\cdot\psi)\\
&  =\widetilde{\tau_{\varphi\cdot\psi}}(f^{A})\text{.}%
\end{align*}
We deduce that
\[
\widetilde{\tau_{\varphi\cdot\psi}}=\varphi\cdot\widetilde{\tau_{\psi}}%
+\psi\cdot\widetilde{\tau_{\varphi}}\text{.}%
\]
That ends the proof.
\end{proof}

For any $\varphi,\psi\in C^{\infty}(M^{A},A)$, we let
\[
\left\{  \varphi,\psi\right\}  _{A}=\widetilde{\tau_{\varphi}}(\psi)\text{.}%
\]
In what follows, we will show that this bracket defines a structure of Poisson
$A$-algebra on $C^{\infty}(M^{A},A)$.

\begin{proposition}
The application
\[
\left\{  ,\right\}  _{A}:C^{\infty}(M^{A},A)\times C^{\infty}(M^{A}%
,A)\longrightarrow C^{\infty}(M^{A},A),(\varphi,\psi)\longmapsto\left\{
\varphi,\psi\right\}  _{A}\text{,}%
\]
is $A$-bilinear and skew-symmetric.
\end{proposition}

\begin{proof}
It is obvious that this application is $A$-bilinear. For any $\varphi\in
C^{\infty}(M^{A},A)$, we verify that the application%
\[
H_{\varphi}:C^{\infty}(M^{A},A)\longrightarrow C^{\infty}(M^{A},A),\psi
\longmapsto\widetilde{\tau_{\varphi}}(\psi)+\widetilde{\tau_{\psi}}(\varphi)
\]
is an $A$-linear derivation. The application
\[
\sigma_{\varphi}:C^{\infty}(M)\longrightarrow C^{\infty}(M^{A},A),f\longmapsto
\widetilde{\tau_{\varphi}}(f^{A})+\widetilde{\tau_{f^{A}}}(\varphi),
\]
is a vector field on $M^{A}$ considered as a derivation of $C^{\infty}(M)$
into $C^{\infty}(M^{A},A)$. As for $f\in C^{\infty}(M)$, we have%
\begin{align*}
H_{\varphi}(f^{A})  &  =\widetilde{\tau_{\varphi}}(f^{A})+\widetilde
{\tau_{f^{A}}}(\varphi)\\
&  =\left(  \widetilde{\sigma_{\varphi}}\right)  (f^{A})\text{.}%
\end{align*}
We deduce, $\cite{bos}$, that
\[
H_{\varphi}=\widetilde{\sigma_{\varphi}}\text{.}%
\]

On the other hand, we have%
\begin{align*}
\left(  \widetilde{\sigma_{\varphi}}\right)  (f^{A})  &  =\left(
\sigma_{\varphi}\right)  (f)\\
&  =\widetilde{\tau_{\varphi}}(f^{A})+\widetilde{\tau_{f^{A}}}(\varphi)\\
&  =\left(  -\widetilde{[ad(f)]^{A}}\right)  (\varphi)+\left(  \widetilde
{[ad(f)]^{A}}\right)  (\varphi)\\
&  =0
\end{align*}
for any $f\in C^{\infty}(M)$. Thus we conclude that $\widetilde{\sigma
_{\varphi}}=0$ i.e. $H_{\varphi}=0$. For any $\psi\in C^{\infty}(M^{A},A)$, we
get%
\[
H_{\varphi}(\psi)=0
\]
i.e.%
\[
\widetilde{\tau_{\varphi}}(\psi)+\widetilde{\tau_{\psi}}(\varphi)=0\text{.}%
\]
Thus
\[
\left\{  \varphi,\psi\right\}  _{A}=-\left\{  \psi,\varphi\right\}
_{A}\text{.}%
\]
As the characteristic is different of $2$, therefore%
\[
\left\{  \varphi,\varphi\right\}  _{A}=0
\]
for any $\varphi\in C^{\infty}(M^{A},A)$.
\end{proof}

\begin{proposition}
For any $\varphi,\psi_{1},\psi_{2}\in C^{\infty}(M^{A},A)$, then%
\[
\left\{  \varphi,\psi_{1}\cdot\psi_{2}\right\}  _{A}=\left\{  \varphi,\psi
_{1}\right\}  _{A}\cdot\psi_{2}+\psi_{1}\cdot\left\{  \varphi,\psi
_{2}\right\}  _{A}\text{.}%
\]
.
\end{proposition}

\begin{proof}
For any $\varphi\in C^{\infty}(M^{A},A)$, as
\[
\widetilde{\tau_{\varphi}}:C^{\infty}(M^{A},A)\longrightarrow C^{\infty}%
(M^{A},A)
\]
is an $A$- linear derivation, then we have%
\begin{align*}
\ \left\{  \varphi,\psi_{1}\cdot\psi_{2}\right\}  _{A}  &  =\widetilde
{\tau_{\varphi}}(\psi_{1}\cdot\psi_{2})\\
&  =\widetilde{\tau_{\varphi}}(\psi_{1})\cdot\psi_{2}+\psi_{1}\cdot
\widetilde{\tau_{\varphi}}(\psi_{2})\\
&  =\left\{  \varphi,\psi_{1}\right\}  _{A}\cdot\psi_{2}+\psi_{1}\cdot\left\{
\varphi,\psi_{2}\right\}  _{A}\text{.}%
\end{align*}
That ends the proof.
\end{proof}

\begin{proposition}
For $f\in C^{\infty}(M)$ and $\varphi\in C^{\infty}(M^{A},A)$ , the
application%
\[
H_{(f^{A},\varphi)}:C^{\infty}(M^{A},A)\longrightarrow C^{\infty}%
(M^{A},A)\text{,}%
\]
defined by
\[
H_{(f^{A},\varphi)}(\psi)=\widetilde{\tau_{\varphi}}\left[  \widetilde
{\tau_{f^{A}}}(\psi)\right]  -\widetilde{\tau_{\psi}}\left[  \widetilde
{\tau_{f^{A}}}(\varphi)\right]  -\widetilde{\tau_{f^{A}}}\left[
\widetilde{\tau_{\varphi}}(\psi)\right]
\]
for any $\psi\in C^{\infty}(M^{A},A)$, is an $A$-linear derivation which is
$zero$.
\end{proposition}

\begin{proof}
In fact for $a\in A$, we have
\begin{align*}
&  H_{(f^{A},\varphi)}(a\cdot\psi)\\
&  =\widetilde{\tau_{\varphi}}\left[  \widetilde{\tau_{f^{A}}}(a\cdot
\psi)\right]  -\widetilde{\tau_{a\cdot\psi}}\left[  \widetilde{\tau_{f^{A}}%
}(\varphi)\right]  -\widetilde{\tau_{f^{A}}}\left[  \widetilde{\tau_{\varphi}%
}(a\cdot\psi)\right] \\
&  =\widetilde{\tau_{\varphi}}\left[  a\cdot\widetilde{\tau_{f^{A}}}%
(\psi)\right]  -(a\cdot\widetilde{\tau_{\psi}})\left[  \widetilde{\tau_{f^{A}%
}}(\varphi)\right]  -\widetilde{\tau_{f^{A}}}\left[  a\cdot\widetilde
{\tau_{\varphi}}(\psi)\right] \\
&  =a\cdot\left(  \widetilde{\tau_{\varphi}}\left[  \widetilde{\tau_{f^{A}}%
}(\psi)\right]  -\widetilde{\tau_{\psi}}\left[  \widetilde{\tau_{f^{A}}%
}(\varphi)\right]  -\widetilde{\tau_{f^{A}}}\left[  \widetilde{\tau_{\varphi}%
}(\psi)\right]  \right) \\
&  =a\cdot H_{(f^{A},\varphi)}(\psi)\text{.}%
\end{align*}
We also have, for $\psi_{1},\psi_{2}\in C^{\infty}(M^{A},A)$,
\begin{align*}
&  H_{(f^{A},\varphi)}(\psi_{1}+\psi_{2})\\
&  =\widetilde{\tau_{\varphi}}\left[  \widetilde{\tau_{f^{A}}}(\psi_{1}%
+\psi_{2})\right]  -\widetilde{\tau_{\psi_{1}+\psi_{2}}}\left[  \widetilde
{\tau_{f^{A}}}(\varphi)\right]  -\widetilde{\tau_{f^{A}}}\left[
\widetilde{\tau_{\varphi}}(\psi_{1}+\psi_{2})\right] \\
&  =\widetilde{\tau_{\varphi}}\left[  \widetilde{\tau_{f^{A}}}(\psi
_{1})+\widetilde{\tau_{f^{A}}}(\psi_{2})\right]  -\widetilde{\tau_{\psi_{1}}%
}\left[  \widetilde{\tau_{f^{A}}}(\varphi)\right]  -\widetilde{\tau_{\psi_{2}%
}}\left[  \widetilde{\tau_{f^{A}}}(\varphi)\right] \\
&  -\widetilde{\tau_{f^{A}}}\left[  \widetilde{\tau_{\varphi}}(\psi
_{1})\right]  -\widetilde{\tau_{f^{A}}}\left[  \widetilde{\tau_{\varphi}}%
(\psi_{2})\right] \\
&  =\ \widetilde{\tau_{\varphi}}\left[  \widetilde{\tau_{f^{A}}}(\psi
_{1})\right]  -\widetilde{\tau_{\psi_{1}}}\left[  \widetilde{\tau_{f^{A}}%
}(\varphi)\right]  -\widetilde{\tau_{f^{A}}}\left[  \widetilde{\tau_{\varphi}%
}(\psi_{1})\right] \\
&  +\widetilde{\tau_{\varphi}}\left[  \widetilde{\tau_{f^{A}}}(\psi
_{2})\right]  -\widetilde{\tau_{\psi_{2}}}\left[  \widetilde{\tau_{f^{A}}%
}(\varphi)\right]  -\widetilde{\tau_{f^{A}}}\left[  \widetilde{\tau_{\varphi}%
}(\psi_{2})\right] \\
&  =H_{(f^{A},\varphi)}(\psi_{1})+H_{(f^{A},\varphi)}(\psi_{2})
\end{align*}
and
\begin{align*}
&  H_{(f^{A},\varphi)}(\psi_{1}\cdot\psi_{2})\\
&  =\widetilde{\tau_{\varphi}}\left[  \widetilde{\tau_{f^{A}}}(\psi_{1}%
\cdot\psi_{2})\right]  -\widetilde{\tau_{\psi_{1}\cdot\psi_{2}}}\left[
\widetilde{\tau_{f^{A}}}(\varphi)\right]  -\widetilde{\tau_{f^{A}}}\left[
\widetilde{\tau_{\varphi}}(\psi_{1}\cdot\psi_{2})\right] \\
&  =\widetilde{\tau_{\varphi}}\left[  \widetilde{\tau_{f^{A}}}(\psi_{1}%
)\cdot\psi_{2}+\psi_{1}\cdot\widetilde{\tau_{f^{A}}}(\psi_{2})\right]
-\left[  \psi_{2}\cdot\widetilde{\tau_{\psi_{1}}}+\psi_{1}\cdot\widetilde
{\tau_{\psi_{2}}}\right]  \left[  \widetilde{\tau_{f^{A}}}(\varphi)\right] \\
&  -\widetilde{\tau_{f^{A}}}\left[  \widetilde{\tau_{\varphi}}(\psi_{1}%
)\cdot\psi_{2}+\psi_{1}\cdot\widetilde{\tau_{\varphi}}(\psi_{2})\right] \\
&  =\widetilde{\tau_{\varphi}}\left[  \widetilde{\tau_{f^{A}}}(\psi_{1}%
)\cdot\psi_{2}\right]  +\widetilde{\tau_{\varphi}}\left[  \psi_{1}%
\cdot\widetilde{\tau_{f^{A}}}(\psi_{2})\right]  -\left(  \psi_{2}%
\cdot\widetilde{\tau_{\psi_{1}}}\right)  \left[  \widetilde{\tau_{f^{A}}%
}(\varphi)\right] \\
&  -\left(  \psi_{1}\cdot\widetilde{\tau_{\psi_{2}}}\right)  \left[
\widetilde{\tau_{f^{A}}}(\varphi)\right]  -\widetilde{\tau_{f^{A}}}\left[
\widetilde{\tau_{\varphi}}(\psi_{1})\cdot\psi_{2}\right]  -\widetilde
{\tau_{f^{A}}}\left[  \psi_{1}\cdot\widetilde{\tau_{\varphi}}(\psi_{2})\right]
\\
&  =\widetilde{\tau_{\varphi}}\left[  \widetilde{\tau_{f^{A}}}(\psi
_{1})\right]  \cdot\psi_{2}+\widetilde{\tau_{f^{A}}}(\psi_{1})\cdot
\widetilde{\tau_{\varphi}}(\psi_{2})+\widetilde{\tau_{\varphi}}\left(
\psi_{1}\right)  \cdot\widetilde{\tau_{f^{A}}}(\psi_{2})\\
&  +\psi_{1}\cdot\widetilde{\tau_{\varphi}}\left[  \widetilde{\tau_{f^{A}}%
}(\psi_{2})\right]  -\psi_{2}\cdot\widetilde{\tau_{\psi_{1}}}\left[
\widetilde{\tau_{f^{A}}}(\varphi)\right]  -\psi_{1}\cdot\widetilde{\tau
_{\psi_{2}}}\left[  \widetilde{\tau_{f^{A}}}(\varphi)\right] \\
&  -\widetilde{\tau_{f^{A}}}\left[  \widetilde{\tau_{\varphi}}(\psi
_{1})\right]  \cdot\psi_{2}-\widetilde{\tau_{\varphi}}(\psi_{1})\cdot
\widetilde{\tau_{f^{A}}}\left(  \psi_{2}\right) \\
&  -\widetilde{\tau_{f^{A}}}\left(  \psi_{1}\right)  \cdot\widetilde
{\tau_{\varphi}}(\psi_{2})-\psi_{1}\cdot\widetilde{\tau_{f^{A}}}\left[
\widetilde{\tau_{\varphi}}(\psi_{2})\right] \\
&  =\widetilde{\tau_{\varphi}}\left[  \widetilde{\tau_{f^{A}}}(\psi
_{1})\right]  \cdot\psi_{2}-\psi_{2}\cdot\widetilde{\tau_{\psi_{1}}}\left[
\widetilde{\tau_{f^{A}}}(\varphi)\right]  -\widetilde{\tau_{f^{A}}}\left[
\widetilde{\tau_{\varphi}}(\psi_{1})\right]  \cdot\psi_{2}\\
&  +\psi_{1}\cdot\widetilde{\tau_{\varphi}}\left[  \widetilde{\tau_{f^{A}}%
}(\psi_{2})\right]  -\psi_{1}\cdot\widetilde{\tau_{\psi_{2}}}\left[
\widetilde{\tau_{f^{A}}}(\varphi)\right]  -\psi_{1}\cdot\widetilde{\tau
_{f^{A}}}\left[  \widetilde{\tau_{\varphi}}(\psi_{2})\right] \\
&  =\left(  \widetilde{\tau_{\varphi}}\left[  \widetilde{\tau_{f^{A}}}%
(\psi_{1})\right]  -\widetilde{\tau_{\psi_{1}}}\left[  \widetilde{\tau_{f^{A}%
}}(\varphi)\right]  -\widetilde{\tau_{f^{A}}}\left[  \widetilde{\tau_{\varphi
}}(\psi_{1})\right]  \right)  \cdot\psi_{2}\\
&  +\psi_{1}\cdot\left(  \widetilde{\tau_{\varphi}}\left[  \widetilde
{\tau_{f^{A}}}(\psi_{2})\right]  -\widetilde{\tau_{\psi_{2}}}\left[
\widetilde{\tau_{f^{A}}}(\varphi)\right]  -\widetilde{\tau_{f^{A}}}\left[
\widetilde{\tau_{\varphi}}(\psi_{2})\right]  \right) \\
&  =H_{(f^{A},\varphi)}(\psi_{1})\cdot\psi_{2}+\psi_{1}\cdot H_{(f^{A}%
,\varphi)}(\psi_{2})
\end{align*}
Thus the application $H_{(f^{A},\varphi)}$ is an $A$-linear derivation.

The application%
\[
\sigma_{(f^{A},\varphi)}:C^{\infty}(M)\longrightarrow C^{\infty}(M^{A},A),
\]
defined by
\[
\sigma_{(f^{A},\varphi)}(g)=\widetilde{\tau_{\varphi}}\left[  \widetilde
{\tau_{f^{A}}}(g^{A})\right]  -\widetilde{\tau_{g^{A}}}\left[  \widetilde
{\tau_{f^{A}}}(\varphi)\right]  -\widetilde{\tau_{f^{A}}}\left[
\widetilde{\tau_{\varphi}}(g^{A})\right]
\]
for any $g\in C^{\infty}(M)$, is a vector field on $M^{A}$. It is obvious that
$\sigma_{(f^{A},\varphi)}$ is linear. For $g,h\in C^{\infty}(M)$, we get
\begin{align*}
\sigma_{(f^{A},\varphi)}(gh)  &  =\widetilde{\tau_{\varphi}}\left[
\widetilde{\tau_{f^{A}}}(gh)^{A}\right]  -\widetilde{\tau_{(gh)^{A}}}\left[
\widetilde{\tau_{f^{A}}}(\varphi)\right]  -\widetilde{\tau_{f^{A}}}\left[
\widetilde{\tau_{\varphi}}(gh)^{A}\right] \\
&  =\widetilde{\tau_{\varphi}}\left[  \widetilde{\tau_{f^{A}}}(g^{A}\cdot
h^{A})\right]  -\widetilde{\tau_{g^{A}\cdot h^{A}}}\left[  \widetilde
{\tau_{f^{A}}}(\varphi)\right]  -\widetilde{\tau_{f^{A}}}\left[
\widetilde{\tau_{\varphi}}(g^{A}\cdot h^{A})\right] \\
&  =\widetilde{\tau_{\varphi}}\left[  \widetilde{\tau_{f^{A}}}(g^{A})\cdot
h^{A}+g^{A}\cdot\widetilde{\tau_{f^{A}}}(h^{A})\right]  -g^{A}\cdot
\widetilde{\tau_{h^{A}}}\left[  \widetilde{\tau_{f^{A}}}(\varphi)\right] \\
&  -h^{A}\cdot\widetilde{\tau_{g^{A}}}\left[  \widetilde{\tau_{f^{A}}}%
(\varphi)\right]  -\widetilde{\tau_{f^{A}}}\left[  \widetilde{\tau_{\varphi}%
}(g^{A})\cdot h^{A}+g^{A}\cdot\widetilde{\tau_{\varphi}}(h^{A})\right] \\
&  =\widetilde{\tau_{\varphi}}[\widetilde{\tau_{f^{A}}}(g^{A})]\cdot
h^{A}+\widetilde{\tau_{f^{A}}}(g^{A})\cdot\widetilde{\tau_{\varphi}}%
(h^{A})+\widetilde{\tau_{\varphi}}(g^{A})\cdot\widetilde{\tau_{f^{A}}}%
(h^{A})\\
&  +g^{A}\cdot\widetilde{\tau_{\varphi}}\left[  \widetilde{\tau_{f^{A}}}%
(h^{A})\right]  -g^{A}\cdot\widetilde{\tau_{h^{A}}}\left[  \widetilde
{\tau_{f^{A}}}(\varphi)\right]  -h^{A}\cdot\widetilde{\tau_{g^{A}}}\left[
\widetilde{\tau_{f^{A}}}(\varphi)\right] \\
&  -\widetilde{\tau_{f^{A}}}[\widetilde{\tau_{\varphi}}(g^{A})]\cdot
h^{A}-\widetilde{\tau_{\varphi}}(g^{A})\cdot\widetilde{\tau_{f^{A}}}%
(h^{A})-\widetilde{\tau_{f^{A}}}(g^{A})\cdot\widetilde{\tau_{\varphi}}%
(h^{A})\\
&  -g^{A}\cdot\widetilde{\tau_{f^{A}}}\left[  \widetilde{\tau_{\varphi}}%
(h^{A})\right] \\
&  =\left(  \widetilde{\tau_{\varphi}}[\widetilde{\tau_{f^{A}}}(g^{A}%
)]-\widetilde{\tau_{g^{A}}}\left[  \widetilde{\tau_{f^{A}}}(\varphi)\right]
-\widetilde{\tau_{f^{A}}}[\widetilde{\tau_{\varphi}}(g^{A})]\right)  \cdot
h^{A}\\
&  +g^{A}\cdot\left(  \widetilde{\tau_{\varphi}}\left[  \widetilde{\tau
_{f^{A}}}(h^{A})\right]  -\widetilde{\tau_{h^{A}}}\left[  \widetilde
{\tau_{f^{A}}}(\varphi)\right]  -\widetilde{\tau_{f^{A}}}\left[
\widetilde{\tau_{\varphi}}(h^{A})\right]  \right) \\
&  =\sigma_{(f^{A},\varphi)}(g)\cdot h^{A}+g^{A}\cdot\sigma_{(f^{A},\varphi
)}(h)\text{.}%
\end{align*}
The application $\sigma_{(f^{A},\varphi)}$ is a vector field on $M^{A}$.

It is obvious that
\[
H_{(f^{A},\varphi)}(g^{A})=\sigma_{(f^{A},\varphi)}(g)
\]
for any $g\in C^{\infty}(M)$. Thus, \cite{bos}, we have
\[
H_{(f^{A},\varphi)}=\widetilde{\sigma_{(f^{A},\varphi)}}\text{.}%
\]
On the other hand,%
\begin{align*}
\widetilde{\sigma_{(f^{A},\varphi)}}(g^{A})  &  =\widetilde{\tau_{\varphi}%
}\left[  \widetilde{\tau_{f^{A}}}(g^{A})\right]  -\widetilde{\tau_{g^{A}}%
}\left[  \widetilde{\tau_{f^{A}}}(\varphi)\right]  -\widetilde{\tau_{f^{A}}%
}\left[  \widetilde{\tau_{\varphi}}(g^{A})\right] \\
&  =\widetilde{\tau_{\varphi}}\left[  \widetilde{\tau_{f^{A}}}(g^{A})\right]
-\widetilde{\tau}_{g^{A}}\left[  \widetilde{\tau_{f^{A}}}(\varphi)\right]
+\widetilde{\tau_{f^{A}}}\left[  \widetilde{\tau_{g^{A}}}(\varphi)\right] \\
&  =\widetilde{\tau_{\varphi}}\left[  \widetilde{\tau_{f^{A}}}(g^{A})\right]
+\left[  \widetilde{\tau_{f^{A}}},\widetilde{\tau_{g^{A}}}\right]  (\varphi)\\
&  =\widetilde{\tau_{\varphi}}\left[  \widetilde{\tau_{f^{A}}}(g^{A})\right]
+\left[  \widetilde{[ad(f)]^{A}},\widetilde{[ad(g)]^{A}}\right]  (\varphi)\\
&  =\widetilde{\tau_{\varphi}}\left(  \left\{  f,g\right\}  ^{A}\right)
+\widetilde{[ad\left\{  f,g\right\}  ]^{A}}(\varphi)\\
&  =-\widetilde{[ad\left\{  f,g\right\}  ]^{A}}(\varphi)+\widetilde
{[ad\left\{  f,g\right\}  ]^{A}}(\varphi)\\
&  =0
\end{align*}
for any $g\in C^{\infty}(M)$. We have, \cite{bos},
\[
\widetilde{\sigma_{(f^{A},\varphi)}}=0
\]
i.e. $H_{(f^{A},\varphi)}=0$.
\end{proof}

\begin{proposition}
For any $\varphi,\psi\in C^{\infty}(M^{A},A)$, then
\[
\left[  \widetilde{\tau_{\varphi}},\widetilde{\tau_{\psi}}\right]
=\widetilde{\tau_{\left\{  \varphi,\psi\right\}  _{A}}}\text{.}%
\]

\end{proposition}

\begin{proof}
For $f\in C^{\infty}(M)$ and $\varphi\in C^{\infty}(M^{A},A)$, as
$H_{(f^{A},\varphi)}=0$, then
\[
H_{(f^{A},\varphi)}(\psi)=0
\]
for any $\psi\in C^{\infty}(M^{A},A)$. Thus, we obtain
\begin{align*}
\widetilde{\tau_{f^{A}}}\left[  \widetilde{\tau_{\varphi}}(\psi)\right]   &
=\widetilde{\tau_{\varphi}}\left[  \widetilde{\tau_{f^{A}}}(\psi)\right]
-\widetilde{\tau_{\psi}}\left[  \widetilde{\tau_{f^{A}}}(\varphi)\right] \\
\left\{  f^{A},\left\{  \varphi,\psi\right\}  _{A}\right\}  _{A}  &  =\left\{
\varphi,\left\{  f^{A},\psi\right\}  _{A}\right\}  _{A}-\left\{  \psi,\left\{
f^{A},\varphi\right\}  _{A}\right\}  _{A}\\
-\left\{  \left\{  \varphi,\psi\right\}  _{A},f^{A}\right\}  _{A}  &
=-\left\{  \varphi,\left\{  \psi,f^{A}\right\}  _{A}\right\}  _{A}+\left\{
\psi,\left\{  \varphi,f^{A}\right\}  _{A}\right\}  _{A}\text{.}%
\end{align*}
i.e.%
\begin{align*}
\left\{  \left\{  \varphi,\psi\right\}  _{A},f^{A}\right\}  _{A}  &  =\left\{
\varphi,\left\{  \psi,f^{A}\right\}  _{A}\right\}  _{A}-\left\{  \psi,\left\{
\varphi,f^{A}\right\}  _{A}\right\}  _{A}\\
\widetilde{\tau_{\left\{  \varphi,\psi\right\}  _{A}}}(f^{A})  &
=(\widetilde{\tau_{\varphi}}\circ\widetilde{\tau_{\psi}})(\left(
f^{A}\right)  -(\widetilde{\tau_{\psi}}\circ\widetilde{\tau_{\varphi}%
})(\left(  f^{A}\right) \\
&  =\left[  \widetilde{\tau_{\varphi}},\widetilde{\tau_{\psi}}\right]  \left(
f^{A}\right)  \text{.}%
\end{align*}
As for any $f\in C^{\infty}(M)$, we have
\[
\left[  \widetilde{\tau_{\varphi}},\widetilde{\tau_{\psi}}\right]  \left(
f^{A}\right)  =\widetilde{\tau_{\left\{  \varphi,\psi\right\}  _{A}}}%
(f^{A})\text{.}%
\]
Therefore, \cite{bos},
\[
\left[  \widetilde{\tau_{\varphi}},\widetilde{\tau_{\psi}}\right]
=\widetilde{\tau_{\left\{  \varphi,\psi\right\}  _{A}}}\text{.}%
\]
That ends the proof.
\end{proof}

We now will show the identity of Jacobi.

\begin{proposition}
For any $\varphi,\psi,\phi\in C^{\infty}(M^{A},A)$, then%
\[
\left\{  \varphi,\left\{  \psi,\phi\right\}  _{A}\right\}  _{A}+\left\{
\psi,\left\{  \phi,\varphi\right\}  _{A}\right\}  _{A}+\left\{  \phi,\left\{
\varphi,\psi\right\}  _{A}\right\}  _{A}=0\text{.}%
\]

\end{proposition}

\begin{proof}
For $\varphi,\psi,\phi\in C^{\infty}(M^{A},A)$, we obtain%
\begin{align*}
&  \left\{  \varphi,\left\{  \psi,\phi\right\}  _{A}\right\}  _{A}+\left\{
\psi,\left\{  \phi,\varphi\right\}  _{A}\right\}  _{A}+\left\{  \phi,\left\{
\varphi,\psi\right\}  _{A}\right\}  _{A}\\
&  =\left\{  \varphi,\left\{  \psi,\phi\right\}  _{A}\right\}  _{A}-\left\{
\psi,\left\{  \varphi,\phi\right\}  _{A}\right\}  _{A}-\left\{  \left\{
\varphi,\psi\right\}  _{A},\phi\right\}  _{A}\\
&  =\widetilde{\tau_{\varphi}}\left[  \widetilde{\tau_{\psi}}(\phi)\right]
-\widetilde{\tau_{\psi}}\left[  \widetilde{\tau_{\varphi}}(\phi)\right]
-\widetilde{\tau_{\left\{  \varphi,\psi\right\}  _{A}}}(\phi)\\
&  =(\left[  \widetilde{\tau_{\varphi}},\widetilde{\tau_{\psi}}\right]
-\widetilde{\tau_{\left\{  \varphi,\psi\right\}  _{A}}})(\phi)\\
&  =0\text{.}%
\end{align*}
That ends the proof.
\end{proof}

Thus, we have shown the following theorem:

\begin{theorem}
If $M$ is a Poisson manifold with bracket $\left\{  ,\right\}  $, then $M^{A}$
is a $A$-Poisson manifold with the bracket%
\[
\left\{  ,\right\}  _{A}:C^{\infty}(M^{A},A)\times C^{\infty}(M^{A}%
,A)\longrightarrow C^{\infty}(M^{A},A),(\varphi,\psi)\longmapsto\left\{
\varphi,\psi\right\}  _{A}=\widetilde{\tau}_{\varphi}(\psi)\text{.}%
\]

\end{theorem}

In this case, we will say that the structure of $A$-Poisson manifold on
$M^{A}$\ defined by $\left\{  ,\right\}  _{A}$ is the prolongation on $M^{A}%
$\ of the structure of Poisson manifold on $M$ defined by $\left\{  ,\right\}
$.

\subsection{Structure of $A$-Poisson manifold on $M^{A}$ when $M$ is a
symplectic manifold}

\begin{proposition}
If $\omega$ is a differential form on $M$ and if $\theta$ is a vector field on
$M$, then
\[
(i_{\theta}\omega)^{A}=i_{\theta^{A}}(\omega^{A})\text{.}%
\]

\end{proposition}

\begin{proof}
If the degree of $\omega$ is $p$, then $(i_{\theta}\omega)^{A}$ is the unique
differential $A$-form of degree $p-1$ such that
\begin{align*}
(i_{\theta}\omega)^{A}(\theta_{1}^{A},...,\theta_{p-1}^{A})  &  =\left[
(i_{\theta}\omega)(\theta_{1},...,\theta_{p-1})\right]  ^{A}\\
&  =\left[  \omega(\theta,\theta_{1},...,\theta_{p-1})\right]  ^{A}%
\end{align*}
for any $\theta_{1},\theta_{2},...,\theta_{p-1}\ \in$ $\mathfrak{X}(M)$. As
$i_{\theta^{A}}(\omega^{A})$ is of degree $p-1$ and is such that%
\begin{align*}
i_{\theta^{A}}(\omega^{A})\left[  \theta_{1}^{A},...,\theta_{p-1}^{A}\right]
&  =\omega^{A}(\theta^{A},\theta_{1}^{A},...,\theta_{p-1}^{A})\\
&  =\left[  \omega(\theta,\theta_{1},...,\theta_{p-1})\right]  ^{A}%
\end{align*}
for any $\theta_{1},\theta_{2},...,\theta_{p-1}\ \in$ $\mathfrak{X}(M)$, we
conclude that $(i_{\theta}\omega)^{A}=i_{\theta^{A}}(\omega^{A})$.
\end{proof}

\begin{proposition}
If $(M,\Omega)$ is a symplectic manifold, then the application
\[
\mathfrak{X}(M^{A})\longrightarrow\Lambda^{1}(M^{A},A),X\longmapsto
i_{X}\Omega^{A},
\]
is an isomorphism of $C^{\infty}(M^{A},A)$-modules.
\end{proposition}

\begin{proof}
Let $X$ be a vector field on $M^{A}$ such that $i_{X}\Omega^{A}=0$. Let
$\xi\in M^{A}$ with origin $x_{0}\in M$. As $\Omega$ is a symplectic form, we
can choose a system of local coordinates $(x_{1},x_{2},....,x_{2n})$ on an
open $U$, $x_{0}\in U$,\ such that
\[
\Omega|_{U}=\underset{i=1}{\overset{n}{%
{\displaystyle\sum}
}}dx_{i}\ \Lambda dx_{i+n}\text{.}%
\]
Thus $\xi\in U^{A}$ and
\[
\Omega^{A}|_{U^{A}}=\underset{i=1}{\overset{n}{%
{\displaystyle\sum}
}}d^{A}(x_{i}^{A})\ \Lambda d^{A}(x_{i+n}^{A})\text{.}%
\]
As $i_{X}\Omega^{A}=0$, by writing $X|_{U^{A}}=\underset{i=1}{\overset{n}{%
{\displaystyle\sum}
}}f_{i}\left(  \frac{\partial}{\partial x_{i}}\right)  ^{A}+\underset
{i=1}{\overset{n}{%
{\displaystyle\sum}
}}f_{i+n}\left(  \frac{\partial}{\partial x_{i+n}}\right)  ^{A}$where
$f_{i},f_{i+n}\in C^{\infty}(U^{A},A)$ for $i=1,2,...,n$, we have%
\begin{align*}
0  &  =\left[  i_{X|_{U^{A}}}\Omega^{A}|_{U^{A}}\right]  \left(
(\frac{\partial}{\partial x_{i}})^{A}\right) \\
&  =-f_{i+n}%
\end{align*}
and
\begin{align*}
0  &  =\left[  i_{X|_{U^{A}}}\Omega^{A}|_{U^{A}}\right]  \left(
(\frac{\partial}{\partial x_{i+n}})^{A}\right) \\
&  =f_{i}%
\end{align*}
for $i=1,2,...,n$. Thus $X|_{U^{A}}=0$. Therefore $X(\xi)=0$. As $\xi$ is
arbitrary, we have $X=0$. The application
\[
\mathfrak{X}(M^{A})\longrightarrow\Lambda^{1}(M^{A},A),X\longmapsto
i_{X}\Omega^{A},
\]
is injective.

Let $\eta\in\Lambda^{1}(M^{A},A)$, $\xi\in M^{A}$ with origine $x_{0}\in U$
where $(U,\varphi)$ is a chart with local coordinates $(x_{1},x_{2}%
,....,x_{2n})$ such that
\[
\Omega|_{U}=\underset{i=1}{\overset{n}{%
{\displaystyle\sum}
}}dx_{i}\ \Lambda dx_{i+n}%
\]
and
\[
\Omega^{A}|_{U^{A}}=\underset{i=1}{\overset{n}{%
{\displaystyle\sum}
}}d^{A}(x_{i}^{A})\ \Lambda d^{A}(x_{i+n}^{A})\text{.}%
\]

By writing
\[
\eta|_{U^{A}}=\underset{i=1}{\overset{n}{%
{\displaystyle\sum}
}}h_{i}d^{A}(x_{i}^{A})+\underset{i=1}{\overset{n}{%
{\displaystyle\sum}
}}h_{i+n}d^{A}(x_{i+n}^{A}),
\]
where $h_{i},h_{i+n}\in C^{\infty}(U^{A},A)$ for $i=1,2,...,n$, we verify that
the vector field
\[
\theta_{U^{A}}=\underset{i=1}{\overset{n}{%
{\displaystyle\sum}
}}h_{i+n}\left(  \frac{\partial}{\partial x_{i}}\right)  ^{A}-\underset
{i=1}{\overset{n}{%
{\displaystyle\sum}
}h_{i}}\left(  \frac{\partial}{\partial x_{i+n}}\right)  ^{A}%
\]
is such that $i_{\theta_{U^{A}}}\Omega^{A}|_{U^{A}}=\eta|_{U^{A}}$. If
$(V,\psi)$ is an other chart around $x_{0}$ with local coordinates
$(x_{1}^{^{\prime}},x_{2}^{^{\prime}},....,x_{2n}^{^{\prime}})$ such that
\[
\Omega|_{V}=\underset{i=1}{\overset{n}{%
{\displaystyle\sum}
}}dx_{i}^{\prime}\ \Lambda dx_{i+n}^{^{\prime}}%
\]
and
\[
\Omega^{A}|_{V^{A}}=\underset{i=1}{\overset{n}{%
{\displaystyle\sum}
}}d^{A}(x_{i}^{^{\prime}A})\ \Lambda d^{A}(x_{i+n}^{^{\prime}A})\text{.}%
\]
We have
\begin{align*}
\eta|_{U^{A}\cap V^{A}}  &  =(\eta|_{U^{A}})|_{U^{A}\cap V^{A}}\\
&  =(i_{\theta_{U^{A}}|_{U^{A}\cap V^{A}}}\Omega^{A})|_{U^{A}\cap V^{A}}%
\end{align*}
and
\begin{align*}
\eta|_{U^{A}\cap V^{A}}  &  =(\eta|_{V^{A}})|_{U^{A}\cap V^{A}}\\
&  =(i_{\theta_{V^{A}}|_{U^{A}\cap V^{A}}}\Omega^{A})|_{U^{A}\cap V^{A}}%
\end{align*}
Thus $\theta_{U^{A}}|_{U^{A}\cap V^{A}}=\theta_{V^{A}}|_{U^{A}\cap V^{A}}$. If
$(U_{i})_{i\in I}$ is a covering of $M$ with such opens, then there exists a
vector field $X$ on $M^{A}$ such that
\[
X|_{U_{i}^{A}}=\theta_{U_{i}^{A}}\text{.}%
\]
We have $\eta=i_{X}\Omega^{A}$ and we conclude that the application
\[
\mathfrak{X}(M^{A})\longrightarrow\Lambda^{1}(M^{A},A),X\longmapsto
i_{X}\Omega^{A},
\]
is surjective.
\end{proof}

\begin{corollary}
When $(M,\Omega)$ is a symplectic manifold, then $(M^{A},\Omega^{A})$ is a
symplectic $A$-manifold.
\end{corollary}

When $(M,\Omega)$ is a symplectic manifold, for any $f\in C^{\infty}(M)$, we
denote $X_{f}$ the unique vector field on $M$ \ such that
\[
i_{X_{f}}\Omega=df
\]
and for any $\varphi\in C^{\infty}(M^{A},A)$, we denote $X_{\varphi}$ the
unique vector field on $M^{A}$, considered as a derivation of $C^{\infty}(M)$
into $C^{\infty}(M^{A},A)$, such that
\[
i_{X_{\varphi}}\Omega^{A}=d^{A}(\varphi)\text{.}%
\]
In this case, we know that
\[
X_{f}=ad(f)\text{.}%
\]
We easily verify that the bracket
\begin{align*}
\left\{  \varphi,\psi\right\}  _{\Omega^{A}}  &  =-\Omega^{A}(X_{\varphi
},X_{\psi})\\
&  =\widetilde{X_{\varphi}}(\psi)
\end{align*}
defines a structure of $A$-Poisson manifold on $M^{A}$.

\begin{proposition}
If $(M,\Omega)$ is a symplectic manifold, for any $f\in C^{\infty}(M)$ then
\[
X_{f^{A}}=(X_{f})^{A}\text{.}%
\]

\end{proposition}

\begin{proof}
The differential $A$-form%
\[
i_{(X_{f})^{A}}\Omega^{A}%
\]
is the unique differential $A$-form of degree $1$ such that%
\begin{align*}
\left[  i_{(X_{f})^{A}}\Omega^{A}\right]  (\theta^{A})  &  =\Omega^{A}%
((X_{f})^{A},\theta^{A})\\
&  =\left[  \Omega(X_{f},\theta)\right]  ^{A}%
\end{align*}
for any $\theta$ $\in$ $\mathfrak{X}(M)$. On the other hand, the differential
$A$-form $i_{X_{f^{A}}}\Omega^{A}$ is of degree $1$ and is such that
\begin{align*}
\left[  i_{X_{f^{A}}}\Omega^{A}\right]  (\theta^{A})  &  =\Omega^{A}(X_{f^{A}%
},\theta^{A})\\
&  =\ \widetilde{\theta^{A}}(f^{A})\\
&  =\ \theta^{A}(f)\\
&  =\left[  \theta(f)\right]  ^{A}\\
&  =\left[  \left(  df\right)  (\theta)\right]  ^{A}\\
&  =\left(  df\right)  ^{A}(\theta^{A})\\
&  =\left[  i_{X_{f}}\Omega\right]  ^{A}(\theta^{A})\\
&  =\ \left[  \Omega(X_{f},\theta)\right]  ^{A}%
\end{align*}
for any $\theta$ $\in$ $\mathfrak{X}(M)$. We conclude that
\[
i_{(X_{f})^{A}}\Omega^{A}=i_{X_{f^{A}}}\Omega^{A}\text{.}%
\]
Thus, we deduce that $X_{f^{A}}=(X_{f})^{A}$.
\end{proof}

We state the following theorem:

\begin{theorem}
If $(M,\Omega)$ is a symplectic manifold, the structure of $A$-Poisson
manifold on $M^{A}$ defined by $\Omega^{A}$ coincide with the prolongation on
$M^{A}$ of the Poisson structure on $M$ defined by the symplectic form
$\Omega$.
\end{theorem}

\begin{proof}
We will show that%
\[
\widetilde{\tau_{\varphi}}=\widetilde{X_{\varphi}}%
\]
for any $\varphi\in C^{\infty}(M^{A},A)$. For any $f\in C^{\infty}(M)$, we
have%
\begin{align*}
\widetilde{X_{\varphi}}(f^{A})  &  =\ \left[  d^{A}\left(  f^{A}\right)
\right]  \left(  X_{\varphi}\right) \\
&  =\left[  i_{X_{f^{A}}}\Omega^{A}\right]  \left(  X_{\varphi}\right) \\
\  &  =-\Omega^{A}\left(  X_{\varphi},X_{f^{A}}\right) \\
&  =-\Omega^{A}\left[  X_{\varphi},(X_{f})^{A}\right] \\
&  =-\left[  i_{X_{\varphi}}\Omega^{A}\right]  \left(  (X_{f})^{A}\right) \\
&  =-(d^{A}\varphi)\left(  (X_{f})^{A}\right) \\
&  =-\widetilde{(X_{f})^{A}}(\varphi)\\
&  =-\widetilde{\left[  ad\left(  f\right)  \right]  ^{A}}(\varphi)\\
&  =\widetilde{\tau_{\varphi}}(f^{A})\text{.}%
\end{align*}
We deduce, \cite{bos}, that
\[
\widetilde{\tau_{\varphi}}=\widetilde{X_{\varphi}}\text{.}%
\]
Therefore, for any $\varphi,\psi\in C^{\infty}(M^{A},A)$, we have
\[
\left\{  \varphi,\psi\right\}  _{\Omega^{A}}=\left\{  \varphi,\psi\right\}
_{A}\text{.}%
\]
That ends the proof.
\end{proof}

\end{document}